\documentclass[]{amsart}
\usepackage[utf8]{inputenc}

\usepackage[english]{babel}
\usepackage{amsmath}
\usepackage{amssymb}
\usepackage{amsthm}

\usepackage[colorlinks=true, allcolors=blue]{hyperref}

\usepackage{tikz}
\usepackage{tikz-cd}
\usepackage{enumitem}
\usepackage{parskip}

\newcommand{\zerodisplayskips}{
\setlength{\abovedisplayskip}{7pt}
\setlength{\belowdisplayskip}{7pt}
\setlength{\abovedisplayshortskip}{2pt}
\setlength{\belowdisplayshortskip}{2pt}
}
\newcommand{\floor}[1]{\left\lfloor #1 \right\rfloor}

\appto{\normalsize}{\zerodisplayskips}
\appto{\small}{\zerodisplayskips}
\appto{\footnotesize}{\zerodisplayskips}

\usepackage{pgfplots}
\pgfplotsset{compat=1.16}

\newtheorem{theorem}{Theorem}
\newtheorem{corollary}{Corollary}
\newtheorem{lemma}[theorem]{Lemma}
\theoremstyle{definition}
\newtheorem{remark}{Remark}
\newtheorem{definition}{Definition}

\newtheorem{conjecture}{Conjecture}
\newtheorem{prop}{Proposition}

\newcommand{\R}{\mathbb{R}}
\newcommand{\N}{\mathbb{N}}

\newcommand{\E}{\mathbb{E}}

\newcommand\norm[1]{\left\lVert#1\right\rVert}
\newcommand\innerprod[1]{\langle#1\rangle}

\usepackage{graphicx, float}
\usepackage{optidef}
\usepackage{braket}
\newcommand{\SET}[1]{\Set{\mskip-\medmuskip #1 \mskip-\medmuskip}}
\begin{document}

\title[]{On the Structure of Bad Science Matrices}

\author[]{Alex Albors}
\address{University of Washington, Seattle, WA 98195, USA}
\email{aalbors@uw.edu}

\author[]{Hisham Bhatti}
\address{ University of Washington, Seattle, WA 98195, USA}
\email{hishamb@uw.edu}

\author[]{Lukshya Ganjoo}
\address{ University of Washington, Seattle, WA 98195, USA}
\email{lganjoo@uw.edu}

\author[]{Raymond Guo}
\address{ University of Washington, Seattle, WA 98195, USA}
\email{rpg360@uw.edu}

\author[]{Dmitriy Kunisky}
\address{ Department of Applied Mathematics and Statistics, Johns Hopkins University, MD 21211, USA}
\email{kunisky@jhu.edu}

\author[]{Rohan Mukherjee}
\address{ University of Washington, Seattle, WA 98195, USA}
\email{rbmuk@uw.edu}

\author[]{Alicia Stepin}
\address{ University of Washington, Seattle, WA 98195, USA}
\email{astepin@uw.edu}

\author[]{Tony Zeng}
\address{Department of Mathematics, University of Washington, Seattle, WA 98195, USA}
\email{txz@uw.edu}

\begin{abstract}
   The bad science matrix problem consists in finding, among all matrices $A \in \mathbb{R}^{n \times n}$ with rows having unit $\ell^2$ norm, one that maximizes
    \begin{align*}
        \beta(A) = \frac{1}{2^n} \sum_{x \in \{-1, 1\}^n} \|Ax\|_\infty.
    \end{align*}
    Our main contribution is an explicit construction of an $n \times n$ matrix $A$ showing that $\beta(A) \geq \sqrt{\log_2(n+1)}$ which is only $18\%$ smaller than the asymptotic rate. We prove that the every entry of any optimal matrix is a square root of a rational number, and we find provably optimal matrices for $n \leq 4$.
\end{abstract}

\maketitle

\section{Introduction}
The bad science matrix problem \cite{steinerberger} is concerned with understanding, for matrices $A \in \mathbb{R}^{n \times n}$ whose rows are normalized to have $\ell^2$-length $1$, the size of the quantity
$$         \beta(A) = \frac{1}{2^n} \sum_{x \in \{-1, 1\}^n} \|Ax\|_\infty.$$
While originally only defined for square matrices in \cite{steinerberger}, its definition naturally makes sense for $A \in \mathbb{R}^{m \times n}$ where one measures the $\infty$-norm of the $2^n$ vertices of the cube which are mapped to $2^n$ vertices in $\mathbb{R}^m$. We note that the problem can also be framed probabilistically, with the objective function becoming
\begin{align*}
    \beta(A) = \mathbb{E} \norm{AX}_{\infty}
\end{align*}
where $X \in \{-1, 1\}^n$ is an $n$-dimensional Rademacher random vector (i.e., with entries drawn independently and uniformly at random from $\{\pm 1\}$). We will use this notation throughout when we mention probabilistic interpretations of various quantities.
The problem is interesting for a variety of different reasons. We list four.

\subsection{Bad Science.} The name is derived from the following thought experiment: we are given a sequence of coin tosses $x \in \left\{-1,1\right\}^n$ and are trying to understand whether they come from a fair coin or not. One natural way of testing is to fix a vector $a_1 \in \mathbb{R}^n$, normalized to $\|a_1\|_2= 1$, and consider the size of the inner product $\left\langle a_1, x \right\rangle$: Hoeffding's inequality guarantees that if, the coin is fair, then the inner product should be small. Similar reasoning suggests that we could take two vectors $\|a_1\|_2= 1=\|a_2\|_2$ and consider both inner products: they should both be small and if one of them is large, that is probably indicative of the coin not being fair. We may also think of this case as two statistical tests being run. Note that, since we do not make any assumptions about $a_1$ and $a_2$, these tests need not be independent. The question now is whether it is possible to construct $n$ vectors $a_1, \dots, a_n$ corresponding to $n$ fair statistical tests such that, for a typical sequence of random coin tosses $x \in \left\{-1,1\right\}^n$, at least one of the statistical tests will typically end up with a large result even though the coin is fair. This motivates the name of ``bad science matrix problem'': the bad scientist is running a battery of (fair) tests to find an abnormal result even though the coins are not abnormal.

\subsection{Rotating the Cube.} Since the rows are normalized to have length $1$, we deduce that the Frobenius norm of the matrix is
$$ \|A\|_F^2 = \sum_{i,j=1}^{n} A_{ij}^2 = n.$$
The Frobenius norm is simultaneously the sum of the squares of the singular values of the matrix and therefore
$$ n = \sum_{i=1}^{n} \sigma_i(A)^2.$$
Geometrically, this means that the matrix cannot be too expansive.  A very simple example of matrices with this property is the identity matrix $\mbox{Id}_{n \times n}$ or, more generally, orthogonal matrices $Q \in \mathbb{R}^{n \times n}$. Given such a matrix $A$, we could now ask where it ends up sending the unit cube $\left\{-1,1\right\}^n$
$$ A \left\{-1,1\right\}^n = \left\{ \sum_{i=1}^{n} \varepsilon_i c_i: \varepsilon_i \in \left\{-1,1\right\} \right\},$$
where $c_1, \dots, c_n \in \mathbb{R}^n$ are the columns of the matrix $A$. Using $a_1, \dots, a_n \in \mathbb{R}^n$ to denote the rows of the matrix $A$, we then have
\begin{align*}
 \frac{1}{2^n} \sum_{x \in \{-1, 1\}^n} \|Ax\|_2^2 &=  \frac{1}{2^n} \sum_{x \in \{-1, 1\}^n} \sum_{i=1}^{n} \left\langle a_i, x \right\rangle^2 = \sum_{i=1}^{n} \frac{1}{2^n} \sum_{x \in \{-1, 1\}^n}  \left\langle a_i, x \right\rangle^2.
\end{align*}
For any arbitrary vector $y \in \mathbb{R}^n$, we have
\begin{align*}
      \frac{1}{2^n} \sum_{x \in \{-1, 1\}^n} \langle y, x \rangle^2 &= \E [\langle y, x \rangle^2] = \E\left[\left(\sum_{i = 1}^{n} y_i x _i\right)^2\right] = \E\left[\sum_{i, j = 1}^{n} y_i y_j x_i x_j \right] \\
      &= \sum_{i, j = 1}^{n} y_i y_j \E[x_i x_j] = \sum_{i = 1}^{n} y_i^2 \underbrace{\E[x_i^2]}_{1} + \sum_{i \neq j} y_i y_j \underbrace{\E[x_i x_j]}_{0} \\
      &= \sum_{i = 1}^{n} y_i^2 = \|y\|^2
\end{align*}
and therefore
$$  \frac{1}{2^n} \sum_{x \in \{1, 1\}^n} \|Ax\|_2^2 = \|A\|_F^2 = n.$$
One way of interpreting the results is that the normalization on the rows shows that $A$ rotates the vertices of the cube to points whose average $\ell^2$-norm remains fixed. The question is now: can these points have the property that the typical largest entry of such a vertex ends up being significantly larger than the average?

\subsection{The Koml\'os Conjecture.} Another motivation is given by the Koml\'os conjecture. It is known (see Banaszczyk \cite{bana3}) that there exists a constant $C>0$ such that for all matrices $A \in \mathbb{R}^{n \times n}$ whose \textit{columns} are of size $\leq 1$ in $\ell^2$, there always exists a sign vector $\varepsilon \in \left\{-1, 1 \right\}^n$ such that
$$ \|A \varepsilon\|_{\ell^{\infty}} \leq C \sqrt{\log{n}}.$$
The open question is whether the $\sqrt{\log{n}}$ term is necessary or whether the result remains true with a universal constant $C$. It is known that if this were true, the constant $C$ cannot be too small: Kunisky \cite{kunisky} has proven that for each $\eta > 0$ there exists a matrix $A \in \mathbb{R}^{n \times n}$ so that
$$ \forall~\varepsilon \in \left\{-1,1\right\}^n \qquad \| A \varepsilon\|_{\ell^{\infty}} \geq 1 + \sqrt{2} - \eta.$$
The bad science matrix problem is somewhat dual: one can think of the Koml\'os problem as the problems of balancing the columns of a matrix to ensure that the final sum is in a small $\ell^{\infty}$-box centered at the origin. Our problem is less a vector balancing problem and more of a `functional-balancing' question: we have $n$ functionals $\left\langle a_i, \cdot \right\rangle$ whose operator norm is 1 and ask whether they can be arranged so that a typical random $\pm 1$ is outside a large $\ell^{\infty}$-box. To further underscore the similarity, an idea of Kunisky \cite{kunisky} shown to be useful in the setting of the Koml\'os problem, will also be useful in the setting of bad science matrices.

\subsection{Extremizers.} One final possible motivation why this functional may be of intrinsic interest are the results: extremal matrices appear to have an interesting structure. We illustrate this with an example taken from \cite{steinerberger}: the currently best known example for $n=5$ is the matrix
$$ A = \frac{1}{2 \sqrt{3}} \left(
\begin{array}{ccccc}
 2 & 2 & 0 & 0 & 2 \\
 -2 & 2 & 0 & 2 & 0 \\
 -2 & 0 & 0 & -2 & 2 \\
 0 & -\sqrt{3} & \sqrt{3} & \sqrt{3} & \sqrt{3} \\
 0 & \sqrt{3} & \sqrt{3} & -\sqrt{3} & -\sqrt{3} \\
\end{array}
\right).$$
This shows that
$$  \max_{A \in \mathbb{R}^{5 \times 5}} \beta(A) \geq \frac{2+ 3\sqrt{3}}{4} \approx 1.799\dots$$
Moreover, the way this matrix acts on the $2^5 = 32$ vertices of the hypercube $\left\{-1,1\right\}^5$ is interesting: 8 of the 32 vertices are sent to vectors with maximal entry 2 (in absolute value) and the remaining 24 vertices are sent to vertices with maximal coordinate $\sqrt{3}$ (in absolute value). We do not know whether this matrix is indeed optimal. Indeed, even for relatively small dimensions, like $n=3$ or $n=4$, rigorously proving that a certain candidate matrix is indeed extremal appears to be a highly nontrivial problem.

\subsection{Known Results.} Relatively little is known for the bad science matrix problem. Steinerberger proved that, as $n \rightarrow \infty$, the `worst' matrix scales as
$$ \beta(A) = (1+o(1)) \cdot \sqrt{2 \log{n}}.$$
The upper bound follows relatively quickly from Hoeffding's inequality and the union bound. As for the lower bound, it is shown that a random matrix with entries $\pm 1/\sqrt{n}$ very nearly saturates the upper bound. The proof is completely non-constructive and does not yield a single example of a truly `bad' bad science matrix. Moreover, some basic numerical experiments suggest that the extremal matrices (at least in low dimensions, say, $n \leq 8$) are quite different from random $\pm 1/\sqrt{n}$ matrices and have interesting behavior: in particular, their entries were noted to be exclusively algebraic numbers like $\sqrt{3}$ or $3/\sqrt{29}$.
 We also note that the bad science matrix problem appears to admit \textit{many} extremal matrices: there is an obvious symmetry in permuting rows and columns but it appears that there is a great degree of non-uniqueness.
Finally, nothing is known about the more general bad science matrix problem for rectangular matrices.

\section{Main Results}
\subsection{Structure Theorem and Consequences.}
Our first main result is a theorem about the structure of extremal matrices. The result proves a strong one-to-one correspondence between the $i^{\text{th}}$ row of an extremal matrix and the vertices of the hypercube $\left\{-1,1\right\}^n$ that are being mapped by the matrix to a point whose coordinates attain their maximum absolute value in the $i^{\text{th}}$ coordinate.

\begin{theorem}[\hyperlink{proof:structure}{Structure Theorem}]\label{thm:Structure}
    Let $A = \begin{bmatrix} a_1 \cdots a_m \end{bmatrix}^\top$ be an $m \times n$ matrix maximizing the value of $\beta(A)$ among all such matrices with rows normalized in $\ell^2$. Introducing the set $W_i$ of vertices of the hypercube that are being mapped to a vector whose largest entry is in the $i^{\text{th}}$ coordinate,
    $$W_i = \SET{x \in \SET{-1, 1}^n \; | \; \norm{Ax}_\infty = a_i^\top x},$$

    then the $i^\text{th}$ row of $A$ is given by
    \begin{align*}
        a_i = \frac{\sum_{x \in W_i} x}{\left\|\sum_{x \in W_i} x\right\|}.
    \end{align*}
\end{theorem}
This structure theorem has a number of immediate consequences. The first consequence is something that was already observed (without proof) in \cite{steinerberger}: entries of extremal entries appear to be either rational numbers or square roots of rational number (indeed, this was essentially in the construction of extremal candidates in \cite{steinerberger}). This is indeed the case and follows easily from the Structure Theorem.

\begin{corollary}\label{cor1}
The entries of an optimal bad science matrix are square roots of rational numbers.
\end{corollary}

We observe that Theorem 1 gives us something much stronger; it shows that extremal entries are really rescaled vectors which themselves are merely the sum of all vertices of a subset of the hypercube. We will use this structural restriction to solve the bad science matrix problem in dimensions $n \leq 4.$

\begin{corollary}[\hyperlink{proof:cor2}{Cases $n \leq 4$}]\label{cor2}
If $A \in \mathbb{R}^{n \times n}$ has normalized rows, then
\begin{align*}
\E \|AX\|_\infty \leq \begin{cases} \sqrt{2}  \qquad &\mbox{when}~ n=2,\\
(\sqrt{2} + \sqrt{3})/2 \qquad &\mbox{when} ~n=3,\\
\sqrt{3} \qquad &\mbox{when} ~n=4. \end{cases}
\end{align*}
These bounds are attained by the matrices
$$ \frac{1}{\sqrt{2}} \begin{bmatrix} 1 & 1 \\ 1 & -1 \end{bmatrix}, \qquad
\begin{bmatrix}
\frac{1}{\sqrt{3}} & \frac{1}{\sqrt{3}} & \frac{1}{\sqrt{3}} \\
\frac{1}{\sqrt{3}} & -\frac{1}{\sqrt{3}} & \frac{1}{\sqrt{3}} \\
-\frac{1}{\sqrt{2}} & 0 & \frac{1}{\sqrt{2}}
\end{bmatrix},
\quad \mbox{and} \quad  \frac{1}{\sqrt{3}}\begin{bmatrix}
            1 & 1 & 1 & 0 \\
            1 & -1 & -1 & 0 \\
            1 & -1 & 1 & 0 \\
            1 & 1 & -1 & 0
        \end{bmatrix}.$$
\end{corollary}

Our proof of Corollary 2 makes extensive use of the computer: the structure theorem can be used to deduce that extremal matrices have to be of a very specific form. One can then, for small values of $n$, explicitly analyze matrices of that size. Thus, in some sense, our proof of Corollary 2 is based on explicit enumeration. However, this scales rather badly in $n$ and even the case $n=5$ appears to be firmly out of reach. It is rather remarkable that we do not currently have a simple way of establishing extremality for small $n$.
We conclude with a very simple observation.
\begin{corollary}
    The number of extremal $n \times n$ matrices is at most $2^{n \cdot 2^n}$.
\end{corollary}
This can be seen very easily: the number of subsets of $\left\{-1,1\right\}^n$ is $\leq 2^{2^n}$ which shows that there are at most this many possibilities for each row. There are $n$ rows bmatrixbmatrixand the inequality follows. Since there is always at least one extremal matrix (by compactness), permutation of rows and columns shows that the number of extremal matrices is at least $n!$ and probably much larger. It could be interesting if this number could be further narrowed down but this appears to be a difficult problem.

\subsection{Construction 1: Lifting Construction.}
One natural question is whether it is possible to explicitly construct examples of matrices $A$ for which $\beta(A)$ is large. In what follows, we present a number of different approaches to this question. The first observation is that from one good example, one can obtain larger good examples: this is somewhat similar to Sylvester's construction of Hadamard matrices and uses a classic lifting trick.

\begin{theorem}[\hyperlink{proof:thm2}{Lifting Theorem}]\label{thm2}
   For any $A \in \mathbb{R}^{n \times n}$ with rows normalized to length 1, the matrix
   $$ B =  \frac{1}{\sqrt{\beta(A)^2 + 1}} \left[\begin{array}{rr}
\beta(A) \cdot A &  \emph{Id}_{n \times n} \\
\beta(A) \cdot A & -\emph{Id}_{n \times n}
\end{array}\right] \,\,\,\, \mbox{satisfies} \,\,\,\, \beta(B) = \sqrt{\beta(A)^2 + 1}.$$
\end{theorem}

In particular, starting with the optimal $2 \times 2$ matrix $A_2$ and applying Theorem 2 repeatedly, we
arrive at an explicit infinite family of matrices $(A_{2^k \times 2^k})_{k=1}^{\infty}$.

\begin{corollary} \label{cor:power-beta} We have
$$ \beta(A_{2^k \times 2^k}) = \sqrt{k+1} = \sqrt{\log_2(n) +1}.$$
\end{corollary}

We note that, in particular,
$$ \beta(A_{2 \times 2}) = \sqrt{2}, \qquad \beta(A_{4 \times 4}) = \sqrt{3}, \qquad \mbox{and} \qquad \beta(A_{8 \times 8}) = 2$$
which matches the currently best known constructions in these dimensions (see \cite{steinerberger}). Moreover, by Corollary \ref{cor2} above, we know that the $2\times 2$ and $4 \times 4$ constructions are optimal.  It stands to reason that the so arising $8 \times 8$ matrix is also optimal but this remains open. Since $\sqrt{2 \log n}/\sqrt{\log_2 n} = \sqrt{2 \log 2} \approx 1.177$, this construction is a factor of approximately 1.177 away from optimal. The family of matrices matches a construction given by Steinerberger \cite{steinerberger} which is itself based on a construction used by Kunisky \cite{kunisky} for the Koml\'os problem. We note that the construction is limited to $n \times n$ matrices where $n=2^k$ is a power of 2. The next two constructions do not suffer from this drawback.

\subsection{Construction 2: Unsatisfiable Trees.} \label{Unsat-trees}
A more general construction can be given using highly balanced binary trees. We quickly recall their definition.
\begin{definition}[Highly balanced binary trees]\label{def:balanced-trees}
    For a fixed integer $n \geq 1$, fill up a binary tree with vertices from left to right until one has $n$ leaves, and finally add an edge that points into the root.
\end{definition}
We label the edges of such a highly balanced binary tree in the following way: edges that point left have label $-1$,  edges that point right have label $1$, and the edge that points to the root has label $1$.  From here, for a leaf $v$, walk along the unique path from the root to $v$. Then the edge labels of this path becomes a row of the matrix this method generates, where if the length of the path is less than $n$, we make the rest of the entries 0. The case $n=4$ is illustrated in Fig. 1. One of our main results is a precise analysis of this construction: as it turns out, the construction performs exactly as well as the Hadamard-style tensor construction when $n$ is a power of 2.

\begin{figure}[h!]
    \centering
    \begin{minipage}{.45\textwidth}
        \centering
    \[
        \begin{tikzcd}
        	&& {} \\
        	&& \bullet \\
        	& \bullet && \bullet \\
        	\bullet & \bullet && \bullet & \bullet
        	\arrow["1"', from=1-3, to=2-3]
        	\arrow["{-1}", from=2-3, to=3-2]
        	\arrow["1"', from=2-3, to=3-4]
        	\arrow["{-1}", from=3-2, to=4-1]
        	\arrow["1", from=3-2, to=4-2]
        	\arrow["{-1}"', from=3-4, to=4-4]
        	\arrow["1"', from=3-4, to=4-5]
        \end{tikzcd}
    \]
    \end{minipage}%
    \begin{minipage}{0.5\textwidth}
        \centering
\begin{align*}
    \left[\begin{array}{rrrr}
        1 & -1 & -1 & 0 \\
        1 & -1 & 1 & 0 \\
        1 & 1 & -1 & 0 \\
        1 & 1 & 1 & 0
    \end{array}\right]
\end{align*}
    \end{minipage}
    \caption{An unsatisfiable tree and the corresponding matrix.}
\end{figure}
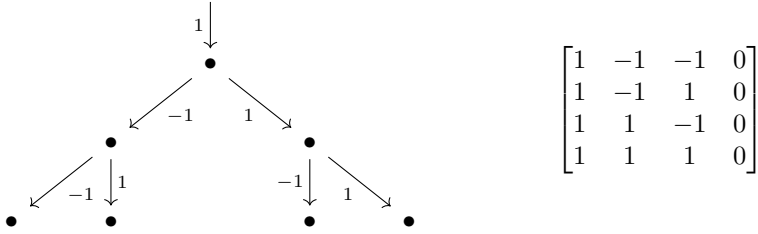

\begin{theorem}[\hyperlink{proof:thm3}{Unsatisfiability Theorem}] \label{beta-bin}
    Let $A \in \R^{n \times n}$ be the matrix whose rows correspond to the $n$ prefixes in the binary tree described above. This construction yields \begin{align*}
        \beta(A) &= \left(2\sqrt{\floor{\log_{2}\left(n\right)+1}}-\sqrt{\floor{\log_{2}n}+2}\right) \\
        &\hspace{1cm} +  \frac{n}{2^{\floor{\log_{2}n}}}\left(\sqrt{\floor{\log_{2}n}+2}-\sqrt{\floor{\log_{2}n}+1}\right)
    \end{align*}
\end{theorem}

When $n=2^k$ is itself a power of 2, the matrices derived from this method satisfy $\beta(A_{2^k \times 2^k}) = \sqrt{k+1}$ which matches Corollary~\ref{cor:power-beta}. However, the construction is more versatile and gives good examples of $n \times n$ matrices where $n$ is not a power of 2.

\subsection{Wide Matrices}
So far, our focus has been on $n \times n$ matrices, however, there is no particular reason for doing so. The $1\times n$ case can be quickly resolved by Jensen's or Khintchine's inequality \cite{haagerup}, and we state it as a useful lemma.
 \begin{lemma}[\hyperlink{proof:lemma4}{Contraction Lemma}] \label{lemma4}
     For any vector $a \in \mathbb{R}^{n}$, we have
     $$ \mathbb{E} |\langle a, X \rangle| = \frac{1}{2^n} \sum_{x \in \left\{-1,1\right\}^n} \left|\left\langle a,x\right\rangle\right| \leq \|a\|_{2}.$$
     Equality is attained when $a$ is a standard unit basis vector.
 \end{lemma}
The $2 \times n$ case is the next natural subject of interest. This case appears to already be remarkably challenging and gives rise to some amusing problems. We believe to have identified the correct answer and verified it numerically in low dimensions, but were unable to prove this rigorously.

 \begin{conjecture} An optimal $2 \times n$ matrix is given by
    \begin{align*}
        A = \frac{1}{\sqrt{2}}
        \begin{bmatrix}
            1 & 1 & 0 & \cdots & 0 \\
            1 & -1 & 0 & \cdots & 0
        \end{bmatrix}
    \end{align*}
\end{conjecture}

The structure theorem immediately suggests the connection to set partitions. While investigating this connection, we came across a question that we were able to resolve and which appears to be of intrinsic interest.

\begin{theorem}[\hyperlink{proof:thm5}{Partition Theorem}] \label{thm5}
  For any $A \subset \{-1, 1\}^n$ we have
  $$\norm{\sum_{x \in A} x}_2 \leq 2^{n-1}$$
  with equality if we choose $A$ to be the set of all vectors whose $i$th coordinate is fixed to be either $1$ or $-1$.
\end{theorem}
\begin{remark}
    Dividing by $2^n$, we see that Theorem $5$ equivalently states that a random vector $X$ that is uniformly distributed in the Boolean cube $\{-1, 1\}^n$ satisfies
    \[
        \big|\big|\mathbb{E}[X \mathbf{1}_{X \in A}]\big|\big|_2 \leq \frac{1}{2}
    \] for any $A \subset \{-1, 1\}^n$. Actually, the same is true not just for the uniform distribution on the Boolean cube, but for any mean zero {isotropic} random vector (a random vector $X$ is isotropic if $\E[X X^\top] = \text{Id}_n$ or equivalently $\E[\innerprod{X, v}^2] = ||v||^2_2$ for all $v \in \mathbb{R}^n$). We state this generalization (for suggesting which we thank an anonymous reviewer) separately below.
\end{remark}

\begin{prop}[\hyperlink{proof:prop1}{Theorem 5 for Centered Isotropic Vectors}] \label{prop1}
    Let $X$ be a mean zero isotropic random vector in $\R^n$ and let $A$ be a measurable subset of $\R^n$. Then
    \[
        \big|\big|\mathbb{E}[X \mathbf{1}_{X \in A}]\big|\big|_2 \leq \frac{1}{2}.    \]
\end{prop}

\subsection{\texorpdfstring{$\ell^p$}{Lg} Variants.}
The $\ell^\infty$ norm leads to fascinating extremizers but is otherwise difficult to deal with. A natural question is whether and in what way things change if we investigate the problem for other $p$. As it turns out, the natural other cases $p=1$ and $p=2$ are not quite as interesting and can be easily resolved. For $2 < p < \infty$ we are unable to compute extremizers but can characterize the two-sided asymptotic behavior.

\begin{prop}[\hyperlink{proof:prop2}{$p=1$ and $p=2$}]\label{prop2} We have
    \begin{align*}
        \max_{A \in \R^{n \times n}}\E \norm{AX}_1 = n
    \end{align*}
    with equality being attained at the identity matrix. Moreover,
    \begin{align*}
        \max_{A \in \R^{n \times n}}\E \norm{AX}_1 = \sqrt{n}
    \end{align*}
    with equality if $A$ is an isometry.
\end{prop}

We note that the problem of finding optimal matrices in the intermediate range $2 < p < \infty$ remains open but Theorem 6 below is a first attempt at characterizing the correct orders of growth. Even though the bounds are not sharp in the cases $p=1,2$, the argument works for all $p \in [1,\infty]$, so we state it in its full generality.
\begin{theorem}[\hyperlink{proof:thm6}{$1 \leq p \leq \infty$}] \label{thm6}
    As before, let $X$ be a random vector with entries drawn independently and uniformly at random from $\{-1, 1\}$. Then,  we have \begin{align*}
    \max_{A \in \R^{n \times n}} \mathbb{E} \|AX\|_p \asymp \begin{cases}
      \sqrt{p} n^{1/p} & \textrm{if } p \leq \log n, \\
      \sqrt{\log n} & \textrm{if } p \geq \log n.
    \end{cases}
    \end{align*}
\end{theorem}

\section{Proofs}
\subsection{Proof of the Structure Theorem}
\begin{proof}[\hypertarget{proof:structure}{Proof of Theorem  \ref{thm:Structure}}]
    Define $W_i$ to be the set of $x \in \SET{\pm 1}^n$ so that $i$ is the smallest index with absolute largest coordinate in $Ax$. By construction, we see that the $W_i$ are closed under $x \mapsto -x$, that they are disjoint, and hence partition $\SET{\pm 1}^n$. Thus we can write:
    \begin{align*}
        \beta(A) = \E \norm{AX}_{\infty} = \frac{1}{2^n}\sum_{x \in \SET{\pm 1}^n} \norm{Ax}_\infty = \frac{1}{2^n}\sum_{i=1}^n \sum_{x \in W_i} |a_i^\top x|
    \end{align*}
    Now take $U_i = \SET{x \in W_i \; | \; |a_i^Tx| = a_i^Tx}$ to be the positive half of $W_i$. For each $x \in W_i$ with $a_i^Tx = 0$, $U_i$ will include both $x$ and $-x$, but we will eventually want $U_i$ to be antipodal. So in this case keep only one of $x$ and $-x$. Then we have that: \begin{align*}
        \sum_{x \in W_i} |a_i^\top x| = 2\sum_{x \in U_i} a_i^\top x.
    \end{align*}
    The idea is to now optimize each of the sums $ \sum_{x \in U_i} a_i^\top x$ separately, since we have removed the independence among the rows from the objective function. One has \begin{align*}
        2 \sum_{x \in U_i} \langle a_i, x \rangle = 2\left \langle a_i, \sum_{x \in U_i} x  \right \rangle
    \end{align*}
    By the Cauchy-Schwarz inequality,
    \begin{align*}
        \left \langle a_i, \sum_{x \in U_i} x  \right \rangle \leq \|a_i\| \left\|\sum_{x \in U_i} x\right\| = \left\|\sum_{x \in U_i} x\right\|
    \end{align*}
    since each of $A$'s rows have norm 1. Now, if $A$ does not have equality for all of the above inequalities, then we could replace $A$'s rows with \begin{align*}
        \frac{\sum_{x \in U_i} x}{\|\sum_{x \in U_i} x\|}
    \end{align*}
    to make the objective function strictly larger. But $A$ is optimal, so we see that $a_i$ is a positive multiple of the above vector, again by Cauchy-Schwarz. Since $a$ has unit norm, we conclude that
    \begin{align*}
        a_i = \pm \frac{\sum_{x \in U_i} x}{\|\sum_{x \in U_i} x\|},
    \end{align*}
    completing the proof.
\end{proof}

Notice that since each of the $x \in U_i$ are vectors on the hypercube, namely have every coordinate of $\pm 1$, $\sum_{x \in U_i} x$ is a vector with integer coordinates. Thus we see that $\norm{\sum_{x \in U_i} x}$ is a square root of an integer, and hence every entry of $a_i$ is a square root of a rational number.

\subsection{Proof of Corollary 2}
\begin{proof}[\hypertarget{proof:cor2}{Proof of Corollary \ref{cor2}}]
This argument uses a computer search to check all possible candidate matrices.
  Theorem \ref{thm:Structure} implies each row $a_i$ of an optimal matrix $A$ is a normalized sum of vectors contained in some subset $V \subseteq \{-1, 1\}^n$, i.e.
  \begin{align*}
    a_i = \frac{\sum_{x \in V}x}{\norm{\sum_{x \in V} x}}.
\end{align*}
We can now check all possible matrices formed by combining rows of this form.
Note, however, that this exhaustive search becomes prohibitive very quickly. Indeed, there are $2^{2^n}$ possible subsets of $\{-1, 1\}^n$, and must consider all possible $n$-tuples of rows of this form. Na\"ively, this yields $(2^{2^n})^n = 2^{n \cdot 2^n}$ matrices to check. For $n = 3$ the search space is of size $2^{24} \approx 1.7 \times 10^{7}$, and for $n = 4$ we have $2^{64} \approx 1.8 \times 10^{19}$ possibilities. We exploit the following symmetries.
\begin{enumerate}[label=(\roman*)]
    \item The $\beta$ value of a matrix is invariant under permutations over rows.
    \item An optimal matrix may not have duplicate rows.
    \item An optimal matrix doesn't have zero rows.
    \item The $\beta$ value of a matrix is invariant under row-wise sign flips
\end{enumerate}
    Observations (i) and (ii) allow us to consider row combinations instead, yielding $\binom{2^{2^n}}{n}$ matrices to check. By observation (iii), we further reduce the search space by deleting the zero rows (for instance the full subset $V = \{-1, 1\}^n$ yields such a row) and by observation (iv) we eliminate duplicate rows equal up to a sign flip (for instance, we keep only one of the rows arising from the subsets $V_1 = \{1\} \times\{-1, 1\}^{n-1}$ and $V_2 = \{-1\} \times \{-1, 1\}^{n-1}$).

    For the $n = 3$ case, we start with $2^{2^3} = 256$ rows, $238$ of which are non-zero. Furthermore, since there may be multiple subsets $V$ leading to the same row, we eliminate redundant ones, yielding $50$ unique rows. Finally, removing redundant rows up to a sign-change reduces the search space to $25$ final rows. This leaves $\binom{25}{3} = 2300$ candidate matrices, which can be easily computed.

    The $n = 4$ case is approached in the same manner.
    Performing the same reductions as before, we go from $2^{2^4} = 65536$ rows to $680$ candidate rows, yielding $\binom{680}{4} = 8830510430$ matrices to check. These can be checked in a rather long but attainable amount of time. For context, a na\"ive sequential search took 72 hours to complete but we suspect it could be vastly parallelized to reduce the runtime.

    However, the $n = 5$ case is prohibitively expensive, since the initial rows $2^{2^5} \approx 4.3 \times 10^9$ cannot even be loaded into memory.
    Similarly, the rectangular bad science problem of dimensions $2 \times 3$, $2 \times 4$ and $3 \times 4$ can be checked, with search spaces of size $\binom{25}{2} = 300, \binom{680}{2} = 230860$ and $\binom{680}{3} = 52174360$, respectively. The first two cases took less than a minute to run, but the $3 \times 4$ took one hour on Google Colab.
    The Jupyter Notebooks containing the row reductions and the exhaustive search, with step-by-step comments, can be found in the Github repository \cite{github}.
\end{proof}

\subsection{Proof of Theorem 2}
\begin{proof}[\hypertarget{proof:thm2}{Proof of Theorem \ref{thm2}}] Fix
\begin{align*}
    \gamma = \frac{\beta(A)}{\sqrt{\beta(A)^2+1}}
\end{align*}
and let $x = [x_1 \; x_2] \in \{-1, 1\}^{2n_0}$, with $x_1 \in \{-1, 1\}^{n_0}$ and $x_2 \in \{-1, 1\}^{n_0}$. Then, $Bx$ can be written as
\begin{align*}
    B x = \begin{bmatrix}
        \gamma A x_1 + \sqrt{1-\gamma^2} x_2 \\
        \gamma A x_1 - \sqrt{1-\gamma^2} x_2
    \end{bmatrix}
\end{align*}
Note that coordinate $j$ of $Bx$ has the form
\begin{align*}
    (Bx)_j = \begin{cases}
        \gamma a_j^\top x_1 + \sqrt{1-\gamma^2}(x_2)_j & \textrm{if } \hspace{5pt} 1 \leq j \leq n_0 \\
         \gamma a_{j-n_0}^\top x_1 - \sqrt{1-\gamma^2}(x_2)_{j-n_0} & \textrm{if } \hspace{5pt} n_0 < j \leq 2n_0
    \end{cases}
\end{align*}
We now note that each term of the form $\gamma \cdot a_j^\top x_1$ contributes to two different coordinates of $Bx$: it does once to coordinate $j$, with the term being added to $\sqrt{1-\gamma^2}(x_2)_j$, and again to coordinate $j + n_0$, with the term being subtracted $\sqrt{1-\gamma^2}(x_2)_j$.
Now, we may bound the absolute value of the coordinates of $Bx$ in a straightforward manner. Letting the index $1 \leq i \leq n_0$ be the one for which $\norm{Ax_1}_{\infty} = |a_i^\top x_1|$ holds, for $1 \leq j \leq n_0$ we have \begin{align*}
    |(Bx)_j| &\leq \max\{|\gamma a_j^\top x_1 + \sqrt{1 - \gamma^2}(x_2)_j|, |\gamma a_j^\top x_1 - \sqrt{1 - \gamma^2}(x_2)_j|\}
        \\ &\leq \gamma |a_j^\top x_1| + \sqrt{1 - \gamma^2}
        \\ &\leq \gamma |a_i^\top x_1| + \sqrt{1 - \gamma^2} = \gamma \norm{Ax_1}_{\infty} + \sqrt{1 - \gamma^2}
\end{align*}
and equally for the second half of the coordinates $|(Bx)_{j + n_0}|$. Thus
$$\norm{Bx}_\infty \leq \gamma \norm{Ax_1}_{\infty} + \sqrt{1 - \gamma^2}.$$
The important observation comes from noting that if the terms $a_j^\top x_1$ and $\sqrt{1 - \gamma^2}(x_2)_j$ agree in sign, then coordinate $j$ will be larger in absolute value than coordinate $j + n_0$, and vice-versa if $a_j^\top x_1$ and $\sqrt{1 - \gamma^2}(x_2)_j$ disagree in sign. With this in mind, we see that one of the choices $j = i$ or $j = i + n_0$ (the one in which the terms $\gamma a_{i}^\top x_1$ and $\sqrt{1 - \gamma^2}(x_2)_i$ or $\gamma a_{i}^\top x_1$ and $-\sqrt{1 - \gamma^2}(x_2)_i$ agree in sign) achieves equality in each step, yielding $\norm{Bx}_\infty = \gamma \norm{Ax_1}_{\infty} + \sqrt{1 - \gamma^2}$.
It is easy to see that summing over all possible $x = \{-1, 1\}^{2n_0}$ translates this result to the $\beta$ setting:
\begin{align*}
  \beta(B) &= \E \norm{BX}_{\infty} =  \frac{1}{2^{2n_0}}\sum_{x \in \{-1, 1\}^{2k}} \|Bx\|_\infty \\ &=  \frac{1}{2^{n_0} \cdot 2^{n_0}} \sum_{x_1, x_2 \in \{\pm 1\}^{n_0}} \gamma \hspace{2pt} \|Ax_1\|_\infty + \sqrt{1 - \gamma^2} = \gamma \beta(A) + \sqrt{1 - \gamma^2}
\end{align*}
 Plugging in $\gamma = \beta(A)/\sqrt{\beta^2(A)+1}$, we obtain
\begin{align*}
    \beta(B) = \frac{1}{\sqrt{\beta(A)^2+1}} + \frac{\beta^2(A)}{\sqrt{\beta(A)^2+1}} = \sqrt{\beta(A)^2+1}.
\end{align*}
Finally, we note that if $A$ has normalized rows then $B$ does too, since $\gamma^2 + 1-\gamma^2 = 1$. This completes the proof.
\end{proof}
\subsection{Proof of Theorem \ref{beta-bin}}
\begin{definition}
    We define $A$ as a set of square matrices which we construct from building a binary tree which is highly balanced as in Definition~\ref{def:balanced-trees}. Every row in the matrix corresponds to some leaf in the tree, and every column in the matrix corresponds to a level of the binary tree. Given this binary tree we yield the corresponding matrix in the following way:
    \begin{itemize}
        \item We let the first column be the all 1's column.
        \item For every path from the root to the leaves: we correspond $A_{i,j}$ to the leaf with index $i \in [n]$ at the $j$'th vertex in the path from the root. If this vertex represents the left child of its parent vertex, we let this entry be $-1$, otherwise we let this entry be $+1$. Entries that do not correspond to vertices in this tree correspond to 0's in the matrix.
    \end{itemize}
\end{definition}
\label{def:third-mat}

\subsection*{Example of binary tree with five leaves}
\begin{align*}
\includegraphics[scale=0.45]{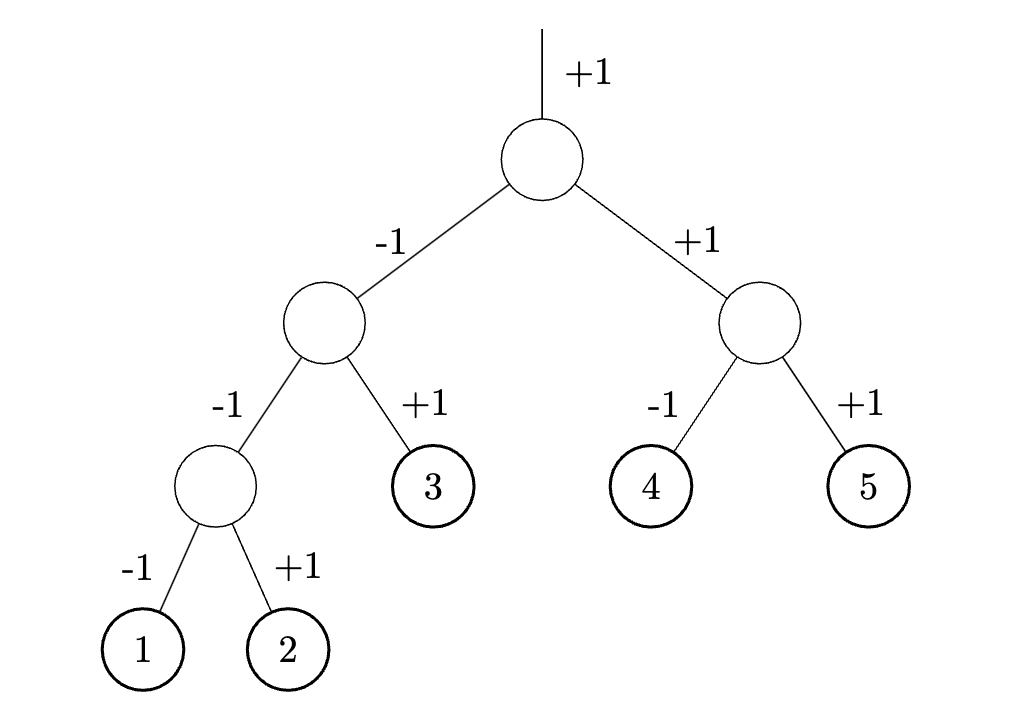}
\end{align*}

We construct the corresponding matrix by mapping every leaf to the row with the same index and the index of every column corresponding to the level in the tree. For example, to construct row 3, we examine the root-to-leaf path of the leaf labeled with 3 which is $\{+1, -1, +1\}$. We set the first 3 entries of the row to match this prefix and set the rest of the entries of the row to 0. We repeat this process for the rest of the rows in this binary tree and before normalization, the corresponding matrix would look as follows:
\[
A_5 = \begin{bmatrix}
+1 & -1 & -1 & -1 & 0 \\
+1 & -1 & -1 & +1 & 0 \\
+1 & -1 & +1 & 0 & 0 \\
+1 & +1 & -1 & 0 & 0 \\
+1 & +1 & +1 & 0 & 0 \\
\end{bmatrix}
\]
We would then normalize the rows of this matrix to discuss its $\beta$ value.
We now use these matrices to prove Theorem 3.

\begin{proof}[\hypertarget{proof:thm3}{Proof of Theorem \ref{beta-bin}}]
Suppose \(2^k < n < 2^{k + 1}\) (deferring the case when \(n = 2^k\) to Corollary~\ref{cor:power-beta}), i.e. \(k = \lfloor\log_2(n)\rfloor\).
Note that all rows of \(A\) are of the form $\{\pm 1\}^{k+1}$ followed by \(n - k - 1\) zeros or $\{\pm 1\}^{k+2}$ followed by \(n - k - 2\) zeros. 
We leverage the fact that if \(|a_i^Tx| = \Vert Ax\Vert_\infty\) for \(x\) in the cube, then \(x = \pm a_i\) entry-wise except for the zero entries of \(a_i\). This is because if we let \(\tilde{a}\) be obtained by truncating the zeros off of \(a_i\), Cauchy-Schwarz tells us that \(\max_{\Vert \tilde{x} \Vert = 1} |\tilde{a}^T\tilde{x}|\) is obtained at \(\tilde{x} = \pm \tilde{a}\).
Therefore, for any \(x \in \{-1, 1\}^n\), \(\Vert Ax \Vert_\infty \) is \(\sqrt{k + 1}\) or \(\sqrt{k + 2}\), corresponding to the length of the matched prefix. The number of possible vectors \(x\) which agree (in the aforementioned manner) with \(a_i\) is given by either \(2^{n - k - 1} \cdot 2\) or \(2^{n - k - 2} \cdot 2\), (two choices for each zero entry, and one last doubling to account for negation). Then, we have that
\begin{align*}
    \beta(A) &= \frac{1}{2^n} \cdot \left[ \sqrt{k + 1} \cdot (2^{k + 1} - n) \cdot 2^{n - k} + \sqrt{k + 2} \cdot (2n - 2^{k + 1}) \cdot 2^{n - k - 1} \right]\\
    &= \sqrt{k + 1} \cdot \frac{2^{k + 1} - n}{2^k} + \sqrt{k + 2} \cdot \frac{2n - 2^{k + 1}}{2^{k + 1}}\\
    &= \left( 2\sqrt{k + 1} - \sqrt{k + 2} \right) + \frac{n}{2^k} \cdot \left(\sqrt{k + 2} - \sqrt{k + 1}\right)
\end{align*}
Replacing \(k\) as \(\lfloor \log_2(n) \rfloor\) recovers the desired expression.
\end{proof}
\subsection{Proof of Lemma 4}
\begin{proof}[\hypertarget{proof:lemma4}{Proof of Lemma \ref{lemma4}}] We may assume without loss of generality that $\|a\|_{\ell^2} = 1$.
    First, recallling that
    \begin{align*}
        \beta(a) = \E \norm{a^\top X}_{\infty} = \frac{1}{2^n} \sum_{x \in \SET{-1, 1}^n} |a^\top x|.
    \end{align*}
    we let $x$ be a Rademacher vector. It remains to prove $\E[|a^\top X|] \leq 1$.
    We start by showing that $\E[XX^\top] = \mbox{Id}_{n \times n}$. Notice that
    \begin{align*}
        \E[XX^\top]_{ij} = \E[X_iX_j] = \begin{cases}
            1 \quad i = j \\
            0 \quad \text{otherwise.}
        \end{cases}
    \end{align*}
    Now by Jensen's inequality,
    \begin{align*}
        \E[|a^\top X|] &= \E\left[\sqrt{(a^\top X)^2}\right] \leq \sqrt{\E[a^\top X X^\top a]} = \sqrt{a^\top \E[XX^\top] a} \\
        &= \sqrt{a^\top a} = 1
    \end{align*}
    since $\norm{a}_2 = 1$. This completes the proof.
\end{proof}

\subsection{Proof of Theorem 5}
We first start with a simple Lemma.
\begin{lemma}
    The set $A \subset \left\{-1,1\right\}^n$ maximizing
    \begin{align*}
       A \rightarrow  \left\| \sum_{a \in A} a \right\|_2
    \end{align*}
    has cardinality $\# A = 2^{n-1}$ and does not contain any antipodal points.
\end{lemma}

\begin{proof}
  For the sake of contradiction, assume the optimal set $A$ contains a set of antipodal points $a,-a \in A$. Then, $A = U \cup \{-a, a\}$ for some set $U \subset \{-1, 1\}^n$. We now consider the term $\norm{\sum_{u \in U} u}^2$. By switching $a$ with $-a$ if necessary, we may without loss of generality assume that $\langle \sum_{u \in U}u , a \rangle \geq 0$. Then, by adding the strictly positive term $\norm{a}^2$ and the nonnegative term $2 \left\langle \sum_{u \in U}u , a \right\rangle$, we  obtain the strict inequality \begin{align*}
    \left\lVert \sum_{u \in U} u \right\rVert^2 <  \left\lVert \sum_{u \in U} u \right\rVert^2 + 2 \left\langle \sum_{u \in U}u , a \right\rangle + \left\lVert a\right\rVert^2 = \left\lVert \sum_{u \in U \cup \{a\}}u \right\rVert^2.
\end{align*}
  Finally, we note that $ \left\lVert \sum_{u \in A} u \right\rVert^2 =  \left\lVert \sum_{u \in U} u \right\rVert^2$ since $a$ and $-a$ cancel out in the sum. This shows the set $A$ is not optimal, a contradiction.
  Having established this, it is easy to see that $\# A = 2^{n-1}$. Indeed, for an optimal set $A$, if $\# A > 2^{n-1}$ the pigeonhole principle ensures two antipodal points, a contradiction. If $\# A < 2^{n - 1}$, we may apply the pigeonhole principle to the complement set $\{-1, 1\}^n \backslash A$ of size $2^n - \# A > 2^{n} - 2^{n - 1} > 2^{n - 1}$, which shares the same objective value as $A$, to determine it is non-optimal.
\end{proof}

\begin{proof}[\hypertarget{proof:thm5}{Proof of Theorem \ref{thm5}}]
Lemma 7 shows that it suffices to consider sets $A$ of size $2^{n-1}$ with no pair of antipodal points (if $a \in A$ then $-a \not\in A$). Working under such assumptions, we let $u = \sum_{a \in A} a$, expand the square of the norm, and use the linearity of the inner product:  \begin{align*}
  \norm{\sum_{a \in A} a}^2 = \sum_{a, b \in A} a^\top b  = \sum_{b \in A} \hspace{2pt} \left(\sum_{a \in A} a\right)^\top b = \sum_{b \in A} u^\top b.
   \end{align*}
   Next, we bound each term by its absolute value and add all the antipodal points into the sum, dividing the total expression by 2 to maintain equality:
   \begin{align*}
    \sum_{b \in A} u^\top b \leq \sum_{b \in A} \left | u^\top b \right |
   = \frac{1}{2} \sum_{b \in \{\pm 1\}^n} \left | u^\top b \right |.
   \end{align*}
   We furthermore assume that $u \neq 0$ since this case is trivially suboptimal. Thus, we normalize $u$ by multiplying and dividing by $\norm{u}$,
   \begin{align*}
    \frac{1}{2} \sum_{b \in \{\pm 1\}^n} \left |u^\top b \right | = \frac{\norm{u}}{2} \sum_{b \in \{\pm 1\}^n} \left |\frac{u}{\norm{u}}^\top b \right |
   \end{align*}
   Since we have a normalized vector and sum over all possible hypercube vectors, we may apply Theorem 4 to obtain
   \begin{align*}
    \frac{\norm{u}}{2} \sum_{b \in \{\pm 1\}^n} \left |\frac{u}{\norm{u}}^\top b \right |  \leq \frac{\norm{u}}{2} \sum_{b \in \{\pm 1\}^n} \left | e_i^\top b  \right |
   \end{align*}
   for all $1 \leq i \leq n$, where $e_i \in \R^n$ is the $i$th standard unit vector. Since $|e_i^\top b| = 1$ for all $b \in \{-1, 1\}^n$, the right hand side is simply
   \begin{align*}
       \frac{\norm{u}}{2} \sum_{b \in \{\pm 1\}^n} \left | e_i^\top b  \right |  = \frac{\norm{u}}{2} \cdot 2^{n} = \norm{u} 2^{n - 1}.
   \end{align*}
   It now suffices to divide both sides by $\norm{u} = \left\lVert \sum_{a \in A} a\right\rVert$ to obtain the upper bound $\norm{\sum_{a \in A} a} \leq 2^{n-1}$. Finally, we determine the optimal sets by asking when equality is achieved. By Theorem 6, we note that the inequalities become equalities if and only if $u/\norm{u} = e_i$. We see that $A = A_i = \{a \in \{-1, 1\}^n: a_i = 1\}$ satisfies
   \begin{align*}
       \frac{\sum_{a \in A} a}{\norm{\sum_{a \in A}a}} = e_i,
   \end{align*}
   and the proof is complete.
\end{proof}

\subsection{Proof of Proposition 1}

We thank an anonymous reviewer for suggesting the proof technique used below.
\begin{proof}[\hypertarget{proof:prop1}{Proof of Proposition \ref{prop1}}]
    Define the random variable
    \[
    \Phi = \begin{cases}
        1 & \text{ if } X \in A, \\
        -1 & \text{ if } X \in \bar{A}.
    \end{cases}
    \] Then $2 X \mathbf{1}_{X \in A} = \Phi X + X$. Since $X$ is mean-zero,
    \[
    2 \E[X \mathbf{1}_{X \in A}] = \E[\Phi X] := v.
    \] Then expanding yields
    \begin{align*}
        \big\|2 \mathbb{E}[X \mathbf{1}_{X \in A}]\big\|^2_2 &= \|\E[\Phi X]\|^2_2 = \innerprod{\E[\Phi X], v} = \E[\Phi \innerprod{X, v}] \\
        &\leq \E|\innerprod{X, v}| \leq \left(\E[\innerprod{X, v}^2]\right)^{1/2} && \text{(Jensen's inequality)} \\
        &= \|v\|_2 = \big\|2 \mathbb{E}[X \mathbf{1}_{X \in A}]\big\|_2.
 && \text{(since $X$ is isotropic)}
    \end{align*}
    Dividing both sides by $\|2\mathbb{E}[X \mathbf{1}_{X \in A}]\|_2$ then completes the proof.
\end{proof}

\subsection{Proof of Proposition 2}
\begin{proof}[\hypertarget{proof:prop2}{Proof of Proposition \ref{prop2}}]
    For $1 \leq p \leq \infty$ and $n \in \N$, we abbreviate
    \[\beta_{n,p} = \max_{A \in \R^{n \times n}} \frac{1}{2^n}\sum_{x \in \{-1, 1\}^n} \norm{Ax}_p =  \max_{A \in \R^{n \times n}} \E \left[\norm{AX}_p\right].\]
    We also note that $\norm{A}_F = \sqrt{\sum_{i,j} A_{ij}^2} = \sqrt{n}$ since $A$ has normalized rows.

    \textbf{The case $p=1$:}
    Taking $a_i$ to be the $i$th row of $A$, we have that,
    \begin{align*}
            \frac{1}{2^n}\sum_{x \in \{-1, 1\}^n} \norm{Ax}_1 =  \frac{1}{2^n}\sum_{x \in \SET{\pm 1}^n} \sum_{i=1}^n \left |a_i^\top x \right | = \sum_{i=1}^n \E \left[|a_i^\top X|\right] \leq n
    \end{align*}
    since $\E[|a_i^TX|] \leq 1$ when $x$ is a vector with $\pm 1$ random entries with equal probability, and $|a_i| = 1$ by Theorem 2.3 (the $1 \times n$ case). The identity matrix achieves this, since $\norm{x}_1 = n$ for every $x \in \SET{\pm 1}^n$.

    \textbf{The case $p=2$:}
    Again by Jensen's inequality, we have that\begin{align*}
        \E \left[\norm{AX}_2\right] &= \E\left[\sqrt{X^\top A^\top A X}\right] \leq \sqrt{\E[X^\top A^\top A X]}.
    \end{align*}
    Then, since if $b \in \R$ then we can treat $b$ as a $1 \times 1$ matrix and have $b = \mathrm{Tr}(b)$, yielding
    \begin{align*}
        \sqrt{\E[X^\top A^\top A X]} &= \sqrt{\E[\mathrm{Tr}(X^\top A^\top AX)]} \\
        &= \sqrt{\E[\mathrm{Tr}(A^\top A XX^\top)]} &&\text{(since } \mathrm{Tr}(ABC) = \mathrm{Tr}(CAB)) \\
        &= \sqrt{\mathrm{Tr}(A^\top A \, \E[XX^\top])} &&\text{(linearity of expectation) } \\
        &= \sqrt{\mathrm{Tr}(A^\top A)} = \sqrt{\mathrm{Tr}(AA^\top)} &&\text{(since } \E[XX^\top] = I)
    \end{align*}
    The entry of $A A^\top$ at position $(i, i)$ is just $a_i^\top a_i = \sum_j A_{ij}^2$ if $a_i$ is the $i$th row of $A$. The sum of all these is then seen to be $\sum_{i,j} A_{ij}^2 = n$ as we showed above. This shows that,
    \begin{align*}
        \E[\norm{AX}_2] \leq \sqrt{n}
    \end{align*}
    Now we shall show that if $A$ is orthogonal then the above becomes an equality. Recall that Jensen's inequality becomes an equality if the random variable is constant. When $A$ is orthogonal, we have that $\norm{Ax}_2^2 = x^\top A^\top A x=x^\top x = n$. Thus in this case $x^\top A^\top Ax$ is constant and the above becomes an equality.
\end{proof}
\subsection{Proof of Theorem 6}
\begin{proof}[\hypertarget{proof:thm6}{Proof of Theorem \ref{thm6}}]

We thank an anonymous reviewer for suggesting an improved proof technique that appears below.

\textbf{Upper bounds.} The simplest case is when $p \geq \log n$,
 which follows immediatelty after observing $n^{1/p} \leq n^{1/\log n} = e$ and using Steinerberger's asymptotic result on $\left\lVert AX \right\rVert _ \infty$ \cite{steinerberger}, which gives  \begin{align*}
  \|AX\|_p \leq n^{1/p}\|AX\|_{\infty} \lesssim \|AX\|_{\infty} \lesssim \sqrt{\log n}.
 \end{align*}
 The case $2 < p \leq \log n$ may be tackled by Khintchine's inequality \cite{haagerup},
 \begin{align*}
     \E |a^\top X|^p \leq \frac{2^{p/2}}{\sqrt{\pi}} \Gamma \left(\frac{p+1}{2}\right) \leq 2^{p/2} (p/2)!
 \end{align*}
 Thus, by Jensen's inequality and Stirling's approximation,
 \begin{align*}
     \E \left[ \norm{AX}_p \right]   &\leq    \left( \sum_{i=1}^n \E|a_i^\top X| \right) ^{1/p}
     \leq \left(n \cdot 2^{p/2}(p/2)!\right)^{1/p} \\
     &\leq C n^{1/p} (p/2)!^{1/p} \leq Cn^{1/p} \frac{(p/2)^{1/2 + 1/p}}{e^{1/2-1/p}} \leq C n^{1/p}(e/2)^{1/p} p^{1/p} p^{1/2} \\
     &\leq Cp^{1/2} n^{1/p},
 \end{align*}
 for some constant $C$ independent of $p$ (which we allow to change from line to line; in this case the constant changed 3 times in 1 line).

\textbf{Lower bounds}. We consider the case where a Rademacher-valued vector $X \in \left\{-1, 1 \right\}^n$ is fixed and the randomness instead comes from the matrix, which has independent and uniformly random $\pm 1/\sqrt{n}$ entries.
 The case $p \geq \log n$ is again straightforward using Steinerberger's asymptotic bound, \begin{align*}
  \mathbb{E} \left\lVert AX \right\rVert _p \geq \mathbb{E} \left\lVert AX \right\rVert _\infty \gtrsim \sqrt{\log n}.
 \end{align*}
 To improve this bound in the case $p \leq \log n$, partition the set $\left\{1, \dots, n\right\}$ into $n/e^p$ disjoint sets $I_j$ of cardinality approximately $e^p$. By concavity of $x \mapsto x^{1/p}$, we apply Jensen's inequality, \begin{align*}
  \mathbb{E} \left\lVert AX \right\rVert _p &\geq \left(\sum_{j = 1}^{n/e^p} \mathbb{E} \sum_{i = 1}^{e^p} \left | \langle a_i, X \rangle \right |^p  \right)^{1/p}  \\
  &\geq \left(\sum_{j = 1}^{n/e^p} \mathbb{E} \max_{i \in I_j} \left | \langle a_i, X \rangle \right |^{p}  \right)^{1/p}
 \end{align*}
 where the second inequality comes from replacing the sum with the maxima. For a fixed $j$ we may apply Steinerberger's asymptotic lower bound again to obtain $\mathbb{E} \max_{i \in I_j} \left | \langle a_i, X \rangle \right |^{p} \gtrsim \sqrt{\log \left | I_j \right | } \gtrsim \sqrt{p}$, hence
 \begin{align*}
  \mathbb{E} \left\lVert AX\right\rVert _p \gtrsim \left(n/e^p\right)^{1/p} \sqrt{p} \gtrsim n^{1/p} \sqrt{p}
 \end{align*}
 Having obtained the desired bounds for this case, we conclude the proof using Fubini-Tonelli:
 \begin{align*}
  \max_A \mathbb{E} \left\lVert AX \right\rVert _p
  \geq \mathbb{E}_{A} \hspace{2pt} \mathbb{E}_X \left\lVert AX \right\rVert _p = \mathbb{E}_{X} \hspace{2pt} \mathbb{E}_A \left\lVert AX \right\rVert _p  \gtrsim \begin{cases}
    \sqrt{p} n^{1/p} & \textrm{if } p \leq \log n \\
    \sqrt{\log n} & \textrm{if } p \geq \log n
  \end{cases}
 \end{align*}
 where $\mathbb{E}_A$ denotes the expectation with respect to the matrices with uniformly random entries in $\pm 1 / \sqrt{n}$, and $\mathbb{E}_X$ to the uniformly distributed Rademacher vectors.
\end{proof}

\end{document}